\documentclass{amsart}
\usepackage{amsfonts,amssymb,latexsym}
\usepackage{amsmath,amsthm}
\usepackage{eucal}
\usepackage[latin1]{inputenc}
\usepackage{color}

\newtheorem{theorem}{Theorem}[section]
\newtheorem{lemma}[theorem]{Lemma}
\newtheorem{proposition}[theorem]{Proposition}

\theoremstyle{definition}

\theoremstyle{remark}

\def\R{{\mathbb R}}
\def\N{{\mathbb N}}

\def\Om{{\Omega}}
\def\span{\hbox{span}}

\numberwithin{equation}{section}
\title{Optimal transportation between hypersurfaces bounding some strictly convex domains}
\author{E. Humbert and L. Molinet}
\begin{document}

\begin{abstract} 
 Let $M,N$ be two smooth compact hypersurfaces of $\R^n$ which bound strictly convex domains equipped with two absolutely continuous measures $\mu$ and $\nu$ (with respect to the volume measures of $M$ and $N$). We consider the optimal transportation from $\mu$ to $\nu$ for the quadratic cost. Let $(\phi:m \to \R,\psi:N \to \R)$ be some functions which  achieve the  supremum in the Kantorovich formulation of the problem and which  satisfy
 $$ \psi (y)  = \inf_{z\in M} \Bigl( \frac{1}{2}|y-z|^2 -\varphi(z)\Bigr);  \\
                 \varphi (x)=\inf_{z\in N} \Bigl( \frac{1}{2}|x-z|^2 -\psi(z)\Bigr).$$
Define for $y \in N$, 
$$\varphi^\Box(y) = \sup_{z\in M} \Bigl( \frac{1}{2}|y-z|^2 -\varphi(z)\Bigr).$$
                 In this short paper, we exhibit a relationship between the regularity of $\varphi^\Box$ and the existence of a solution to the Monge problem. 
                 \end{abstract}

\maketitle
 Let 
 $M$ and $N$ be two smooth compact hypersurfaces of $\R^n$, $n \geq 2$, which are the boundary of some strictly convex domains. In the present paper, we study the existence of a solution of Monge Problem when considering the optimal transport with quadratic cost between two measures $\mu$ and $\nu$ respectively supported in $M$ and $N$. This situations has been already studied: see  \cite{GMC}.
In the whole paper, we assume that $\mu$ and $\nu$ have the form $\mu= f(x) dv(x)$ and $\nu=g(y) dv(y)$ where $f,g$ are some non-zero nonnegative continuous functions on $M$ and $N$ and where $dv$ stands for the volume measures on $M$ and $N$. 
The quadratic cost is defined for all $x,y \in \R^n$ by 
$c_2(x,y):=\frac{1}{2} |x  - y|^2$.
Here, $| \cdot|$ denotes the standard norm associated to the canonical scalar product of $\R^n$. For all $x,y \in \R^n$, the scalar product of $x$ and $y$ will be  denoted by $x \cdot y$.

\noindent  The standard formulation of the optimal transport from $\mu$ to $\nu$ for the quadratic cost is   
$$T_0:= \inf_{\pi \in \Pi'(\mu,\nu)} I_0'(\pi)$$
where $\Pi'(\mu',\nu')$ is the set of probabilily measures $\pi(\cdot,\cdot)$ on $\R^n \times \R^n$ such that
$$\pi(\R^n, \cdot) = \nu \hbox{ and } \pi( \cdot, \R^n)= \mu$$
and for all  $\pi \in \Pi'(\mu',\nu')$, 
$$ I_0'(\pi):= \int_{\R^n \times \R^n} c_2(x,y) d\pi(x,y).$$  
As easily checked, this is an equivalent formulation to write 
\begin{eqnarray} \label{T=inf} 
 T_0:= \inf_{\pi \in \Pi(\mu,\nu)} I_0(\pi)
\end{eqnarray}
where $\Pi(\mu,\nu)$ is the set of probabilily measures $\pi(\cdot,\cdot)$ on $M \times N$ such that
$$\pi(M, \cdot) = \nu \hbox{ and } \pi( \cdot, N)= \mu$$
and for all  $\pi \in \Pi(\mu,\nu)$, 
$$ I_0(\pi):= \int_{M \times N} c_2(x,y) \, d\pi(x,y).$$

\noindent By the Monge-Kantorovich duality (see for instance, \cite{K}, \cite{V}), one has 
\begin{eqnarray} \label{T=sup}
 T_0= \sup_{(\varphi, \psi) \in \Omega} J_0(\varphi,\psi)
\end{eqnarray}
where $\Omega$ is the set of couples of functions $(\varphi, \psi) \in C^0(M) \times C^0(N)$ such that 
$$\varphi(x)  + \psi(y) \leq  c_2(x,y) $$ 
for all $(x,y) \in M \times N$ and   where 
$$J_0(g,h)= \int_{M} g(x) d\mu(x) + \int_{N} h(y) d\nu(y).$$ 
for all $(g,h)\in C^0(M) \times C^0(N)$.

\noindent Actually, in many situations, one can also show the uniqueness of $\pi$ (see \cite{AKMC}). 

\noindent It is standard to prove that (see for instance \cite{AKMC} for references): 
\begin{enumerate} 
 \item the infimum in (\ref{T=inf}) is attained by some probability measure $\pi$, which is called a transference plan; 
 \item the supremum in (\ref{T=sup}) is attained by some couple of functions $(\varphi,\psi)$ which satisfy for all $x\in M$, $y \in N$:
 \end{enumerate}

\begin{equation} \label{mnconvex}
 \psi (y)  = \inf_{z\in M} \Bigl( c_2(z,y) -\varphi(z)\Bigr);  \\
                 \varphi (x)=\inf_{z\in N} \Bigl( c_2(x,z) -\psi(z)\Bigr)
   \end{equation}
In the whole paper, if $\eta:M \to \R$, we will note for all $y \in N$ 
\begin{eqnarray} \label{def*}
 \eta^* (y)  = \inf_{x\in M} \Bigl( c_2(x,y) -\eta(x)\Bigr). 
\end{eqnarray}
In the same way, if 
$\eta :N \to \R$, we will note for all $x \in M$ 
\begin{eqnarray} \label{def*2}
 \eta^* (x)  = \inf_{y\in N} \Bigl( c_2(x,y) -\eta(y)\Bigr),
\end{eqnarray}
so that $\varphi^*= \psi$ and $\psi^*= \varphi$.  \\

\noindent An important question about   \eqref{T=inf} is the following  : are the transference plans associated with \eqref{T=inf} supported in a graph ? Indeed, a positive answer to this question would ensure the existence of a solution to the famous Monge problem (see, for instance, \cite{V} for some explanations).
Unfortunately, the answer is no in its full generality. Gangbo and McCann \cite{GMC} could construct some counter examples. Even worse: numerical computations indicate that this is not true either in the simplest situation when $M=N=S^1$ (for $n \geq 1$, $S^n$ denotes the unit sphere of $\R^{n+1}$) and when $\mu$ and $\nu$ have smooth positive densities (see \cite{CMCN}).

  In this paper we try to give some simple criteria that would imply a positive answer. Before, stating our result, we need to introduce some definitions.   \\


\noindent For any fonction $\Theta:M \to \R$, we  set  for all $y \in N$ 
\begin{eqnarray}
 \Theta^\Box (y)  = \sup_{x\in M} \Bigl( c_2(x,y) -\Theta(x)\Bigr).
\end{eqnarray}
and in the same way, if 
$\Theta:N \to \R$, we   set   for all $x \in M$ 
\begin{eqnarray}
 \Theta^\Box(x)  = \sup_{y\in N} \Bigl( c_2(x,y) -\Theta(x)\Bigr).
\end{eqnarray}

 \noindent While, as proven is \cite{GMC}, the function $\varphi$ is always  $C^1$, the function $\varphi^\Box$ has no reason to be $C^1$ in general but surprisingly, its regularity is directly related to the question above. 
More precisely, our main result is 

\begin{theorem} \label{main}
 For all $y \in N$, define $\Theta_y:M \to \R$ by 
 $$\Theta_y (x) =  c_2(x,y) -\varphi(x)$$
 so that, for $y \in N$, $\varphi^*(y) = \inf_M \Theta_y$ and $\varphi^\Box(y)= \sup_M \Theta_y$. 
 Then, the following assertions are equivalent:
 \begin{enumerate}
  
  \item $\varphi^\Box$ is $C^1$;
  \item for all $y \in N$, the function $\Theta_y$ has exactly two critical points: its minimum and it maximum.  
 \end{enumerate}
If one of the  assertions  above is true then $\varphi^{\Box \Box} = \varphi^{**}= \varphi$. Moreover, the support of $\pi$ is contained in a graph. 
\end{theorem}

 \noindent The idea of the proof is as follows. Let 
 $\Gamma$ be the set of points $x \in M$ which are the maximum of a function $\Theta_y$ for some $y \in N$. There is then two crucial observations: 
 \begin{enumerate} 
  \item If $(x,y)$ belongs to the support of $\pi$ and if $x \in \Gamma$, then $y$ is unique. This implies that if $\Gamma=M$ then the support of $\pi$ is contained in a graph.
  \item Let $x \in \Gamma$,  $y \in N$ such that $x$ is the maximum of $\Theta_y$. Then, $y$ is  unique.  This allows to define a map 
 $T: \Gamma \to N$ such that $T(x)=y$. The main argument of the proof is to show that under the assumptions of Theorem \ref{main}, $T$ is actually a homeomorphism, which implies that $\Gamma=M$ and allows to conclude.
 \end{enumerate}
 
 \noindent 
 Even if assumptions 1) or 2) are not easy to check, we think that this theorem gives a new point of view, that we hope useful, to this Monge problem. 
 For convenience of the reader, we stated all the results which seem of particular interest to us in Propositions \ref{main1}, \ref{main2}, \ref{main3}. Theorem  \ref{main} is a direct consequence of these propositions. 
\section{Proof of Theorem \ref{main}} 
\subsection{Notations and Preliminaries}
 We  keep the notations of the introduction: $(\varphi, \psi)$ is a couple of functions maximizing the problem (\ref{T=sup}). By Gangbo and McCann \cite{GMC}, these functions are $C^1$. Indeed, in their paper Section 3, they show that the convex functions they study are tangentially differentiable and that these tangent differentials are continuous on $M$ and $N$. Here, the function we consider are the same functions restricted to $M$ and $N$ and are hence $C^1$. Notice that the proof of this fact is far to be obvious.  In addition, the functions $\varphi,\psi$ satisfy $\varphi^*=\psi$ and $\psi^* = \varphi$. 
We recall that for all $y \in N$, we defined $\Theta_y:M \to \R$ by 
 $$\Theta_y (x) =  c_2(x,y) -\varphi(x).$$
 In the same way, if $x \in M$, we define $\Theta_x:N \to \R$ by
 $$\Theta_x (x) =  c_2(x,y) -\psi(x).$$
For $x \in M$, $y \in N$, we introduce the sets:
$$\Om_x:= \{ z\in N \, /  \theta_x(z)=\inf_{z'\in N} \theta_x(z')=\varphi^* (z)= \psi(z) \}$$
and  
$$\Om_y:= \{ z\in M \, /  \theta_y(z)=\inf_{z'\in M} \theta_y(z')=\psi^* (z)= \varphi(z) \}.$$
Note that, by compacity of $M$ and $N$,  these sets are non empty. 
Note also that 
$$\Om_x:= \{ z\in N \, /  \varphi(x) + \psi(z) = c_2(x,z) \} \; \hbox{ and } \; \Om_y:= \{ z\in M \, /  \varphi(z) + \psi(y) = c_2(z,y) \}$$
which has the immediate consequence that 
\begin{eqnarray} \label{equivom}
 y\in  \Omega_x \Leftrightarrow x\in \Omega_y \Leftrightarrow \varphi(x) + \psi(y) = c_2(x,y).
\end{eqnarray}

\noindent Now, proving that the support of $\pi$ is contained in a graph of a continuous preserving map $\alpha:M \to N$ (resp. $\alpha:N \to M$) is reduced to proving that for all
$x \in M$ (resp. $y \in N$), the set $\Om_x$ (resp. $\Om_y$) contains exactly one point. 
Indeed, by (\ref{T=inf}) and (\ref{T=sup}), one has
$$\int_{M \times N} c_2(x,y) d\pi(x,y) = \int_M \varphi(x) d\mu(x) + \int_N \psi(y) d\nu (y)$$
which can be rewritten, since the marginals of $\pi$ are $\mu$ and $\nu$, by

$$ \int_{M \times N} \left(c_2(x,y) - \varphi(x) - \psi(y) \right)  d\pi(x,y) = 0.$$
Since $c_2(x,y) - \varphi(x) - \psi(y) \geq 0$, one has identically on the support of $\pi$:
$$c_2(x,y) = \varphi(x) + \psi(y)$$
and hence $x \in \Om_y$ or $y \in \Om_x$. \\

\noindent For any $ y\in N $ we denote by $ n_N(y) $ the unitary normal outside vector to $ N $ at $ y$ and we define the line $ D_y $ by $$ D_y=y-\nabla \psi(y)+\span(n_N(y)) \; .$$ 
Similarly, for $ x\in M $ we define the line $ D_x $ by 
$$D_x=x-\nabla \varphi(x)+\span(n_M(x)) \; .$$
An easy computation of the derivative of $\Theta_y$ and $\Theta_x$ shows that if $x \in M, y \in N$, then 

\begin{eqnarray} \label{critical}
y \in D_x \Leftrightarrow \nabla \Theta_y(x) = 0 \; \hbox{ and } \; 
x \in D_y \Leftrightarrow \nabla \Theta_x(y)= 0.
\end{eqnarray}

\noindent Finally, we will use several times the following Lemma:  
\begin{lemma} \label{nonconstant}
 For all $x \in M$, $y \in N$, the functions $\Theta_x$ and $\Theta_y$ are never constant. 
\end{lemma}
\begin{proof}
 Assume for instance that $\Theta_y$ is constant. Then for any $x \in M$, $x \in \Om_y$ and hence, by (\ref{equivom}), $y \in \Om_x$. By (\ref{critical}), $x \in D_y$ which implies that $M$ is contained in the right line $D_x$. This is impossible and Lemma \ref{nonconstant} follows.  
\end{proof}

\subsection{Properties of $\varphi^\Box$}
In the proof on Theorem \ref{main}, we will use some basic properties of $\varphi^{\Box}$. Many of them are very standard. We first recall its definition: 
$$\varphi^\Box(y)  = \sup_{x\in M} \Bigl( c_2(x,y) -\varphi(x)\Bigr).$$
 For convenience, for any $y \in N$, we will denote by $\Om_y^\Box$ the set of points of $M$ achieving the maximum in the definition of $\varphi^\Box$. 
 We collect the properties we will need in the following Proposition:
\begin{proposition} \label{varphibox}

 \begin{enumerate}
 \item For all $x\in M$, t$y \in N$, one has $\varphi(x) + \varphi^\Box(y)  \geq c_2(x,y)$ with equality if and only if $x \in \Om^\Box_y$; 
 \item $\varphi^{\Box \Box} \leq \varphi$;
 \item $\varphi^{\Box}$ is Lipschitz;
  \item $\varphi^{\Box \Box \Box} = \varphi^\Box$; 
  \item Let $y \in M$. Assume that there is only one point such $x_y$ such that 
  $$\varphi^\Box(x)  = c_2(x_y,y) -\varphi(x_y)$$
i.e. the supremum in the definition of $\varphi^\Box$ is attained at only one point, then $\varphi^\Box$ is differentiable at $y$.  
  \end{enumerate}
\end{proposition}
\begin{proof} 
 \noindent We start by proving 1). Let $x \in M$ and $y \in N$. It holds that 
  
 \begin{eqnarray*}
  \begin{aligned} 
   \varphi(x) + \varphi^\Box(y) & = \varphi(x) + \sup_{z\in M} \Bigl( c_2(z,y)  -\varphi(z)\Bigr) \; \\
 & \geq  \varphi(x)  + (  c_2(x,y)  -\varphi(x))= c_2(x,y).
 \end{aligned}
 \end{eqnarray*}
 The inequality above becomes an equality if and only if $x \in \Om^\Box_y$. This proves 1). \\

 \noindent Let us now deal with 2). 
Let $x \in M$. By definition:
$$\varphi^{\Box \Box}(x) = \sup_{y \in N} (c_2(x,y) - \varphi^\Box(y)).$$
By compacity of $N$, there exists $y_x \in M$ such that 
 \begin{eqnarray} \label{boxbox2} 
  \varphi^{\Box \Box}(x) =  c_2(x,y_x) - \varphi^\Box(y_x).
 \end{eqnarray}
By definition, one also have
 $$\varphi^\Box(y_x) = \sup_{z \in M} (c_2(z,y_x) - \varphi(z)).$$
 and, setting $z=x$, one has 
$$ \varphi^{\Box} (y_x) \geq c_2(x,y_x) - \varphi(x).$$
Together with (\ref{boxbox2}), this gives 2). \\

 \noindent Let us prove 3). Let $y,z \in M$, $x_y \in \Om^\Box_y$ and $x_z \in \Om^\Box_z$. 
 We prove that 
 \begin{eqnarray} \label{lips} 
  \begin{split} 
 \frac{1}{2} (z+y - 2x_y) \cdot  (z-y) 
  &\geq \varphi^\Box(z) - \varphi^\Box(y)  \\ 
 & \geq \frac{1}{2} (z +y - 2x_z) \cdot (z -y).
 \end{split}
 \end{eqnarray}
 
\noindent The definition of $\Omega^\Box_z$ implies that 
 $$\varphi^\Box(z) = c_2(x_z,z) - \varphi(x_z).$$ 
 The construction of $\varphi^\Box$ implies that 
 $$ \varphi^\Box(y) \geq  c_2(x_z,y)  - \varphi(x_z).$$
 Observing that 
 $$\frac{1}{2} (z +y - 2 x_z) \cdot (z -y)= c_2(x_z,z) -c_2(x_z,y),$$
 this provides the right inequality of (\ref{lips}). The left inequality is proven in the same way. 
 Now, observe that since $M,N$ are compact, there exists a constant  $C >0$ independant of $y,z$ such that 
 $$\frac{1}{2} |z +y - 2x_z| \leq C \; \hbox{ and } \frac{1}{2} |z +y - 2x_y| \leq C.$$
 Using that
 $|z -y|$ is less  than the geodesic distance on $N$, we immediatly deduce that $\varphi^\Box$ is lipschitz. 
 Note that this implies that $\varphi^\Box$ is continuous and by Rademacher's Theorem, is differentiable almost everywhere.\\

 \noindent We now prove 4). 
 By Point 2)
 $$\varphi^{\Box \Box} \leq \varphi.$$  In particular, for all $y \in N$, 
 $$\varphi^{\Box \Box \Box}(y) = \sup_{x \in M} ( c_2(x,y) - \varphi^{\Box \Box} (x) ) \geq   
 \sup_{x \in M} (c_2(x,y)  - \varphi(x) )= \varphi^\Box(y).$$
For all $x \in M$, $y \in N$, we also have as in  Point 1) 
$$\varphi^{\Box \Box}(x) + \varphi^\Box(y) \geq c_2(x,y).$$ 
Then as in Point 2), 
$\varphi^{\Box \Box \Box} =(\varphi^{\Box})^{\Box\Box} \leq \varphi^\Box$. 
This shows 4). \\

\noindent Let us finish by proving 5). Let $y \in M$ and assume that $\Om^\Box_y$ is reduced to one point $x$. 
 Let $(z_k)$ be a sequence of points of $N$ tending to $y$. For all $k$, choose $x_k \in \Omega^\Box_{z_k}$. By compacity of $M$, one can assume that $x_k$ converges to some  $x' \in M$.  
 The definition of $\Omega^\Box_{z_k}$ and Point 1) implies that 
$$\varphi(x_k) + \varphi^\Box(z_k) = c_2(x_k,z_k).$$
\noindent By continuity of $\varphi$ and $\varphi^\Box$, we obtain as $z_k$ tends to $y$, 
$$\varphi(x') + \varphi^\Box(y) = c_2(x',y)$$
 which proves that $x' \in \Omega^\Box_y$ and hence $x'=x$
Using \eqref{lips}, we have 
\begin{eqnarray} \label{subdif1}
  \hspace*{10mm}  \frac{1}{2}(z_k-y) \cdot (z_k+y - 2 x_k)  \geq \varphi^\Box(z_k) - \varphi^\Box(y) \geq\frac{1}{2}(z_k-y) \cdot (z_k+y - 2 x).
 \end{eqnarray}
 Until the end of the proof, the notation $o_k$ will stand for a term which is $o(|z_k - y|)$. 
Since $x_k$ tends to $x$, we have 
$$x_k \cdot (z_k -y) = x \cdot (z_k -y) + o_k.$$
When $z_k$ is close to $y$,  
$$ z_k - y = P_y(z_k-y) + o_k,$$
where $P_y$ denotes the orthogonal projection onto the tangent space $T_yN$. 
Coming back to \eqref{subdif1}, we obtain that 
$$\frac{1}{2} P_y (z_k-y) \cdot (z_k +y - 2x)  + o_k \geq \varphi^\Box(z_k) - \varphi^\Box(y) \geq \frac{1}{2}P_y (z_k-y) \cdot (z_k +y - 2x)  + o_k .$$
Since $P_y$ is self-adjoint, this yields 
$$\frac{1}{2} (z_k-y) \cdot P_y  (z_k +y - 2x)  + o_k \geq \varphi^\Box(z_k) - \varphi^\Box(y) \geq \frac{1}{2} (z_k-y) \cdot P_y (z_k +y - 2x)  + o_k.$$
 Noticing that
$$ \lim_k P_y  (z_k +y - 2x) = 2 P_y (y  - x)$$ 
and setting  $v:= P_y (y  - x) $, it follows that
$$(z_k-y) \cdot  v + o_k \geq \varphi^\Box(z_k) - \varphi^\Box(y) \geq (z_k -y)\cdot  v  + o_k.$$
This ensures  that for any sequence $z_k$ tending to $y$, one can extract a subsequence such that 
$$   \varphi^\Box(z_k) - \varphi^\Box(y)  - v \cdot (z_k-y)= o_k.$$
Since when $z_k$ tends to $y$, $(z_k-y)$ is equivalent to the geodesic distance from $y$ to $z_k$ in $N$, this proves that $\varphi^\Box$ is differentiable and that $\nabla \varphi^\Box(y) = v$ which completes the proof of Proposition \ref{varphibox}. 
\end{proof}

\subsection{Proof of Theorem \ref{main}}

\noindent 
We define 
$$
\Gamma:=\{x\in M\, / \, \exists y\in N\, , \, \theta_y(x)=\sup_{x\in M} \theta_y(x)=\varphi^\Box (y) \} 
$$
 
\noindent The first observation is the following:

 \begin{proposition} \label{main1}  {\em (Properties of the set $\Gamma$)}  
 \begin{enumerate} 
  \item The set $\Gamma$ is closed;
  \item If $x \in \Gamma$ then $ \# \Omega_x=1 $.
In particular, if $\Gamma =M$,  the support of $\pi$ is contained in a graph. 
\item For all $x \in \Gamma$, one has $\varphi^{\Box \Box}(x) = \varphi(x) = \varphi^{**}(x)$. 
 \end{enumerate}
  \end{proposition}
 \begin{proof}
Let us first show that $\Gamma$ is closed: let $(x_n) \subset \Gamma$ be such that $ x_n\to x $ in $ M$.  There exists $(y_n) \subset N $ such that for all $ n\in \N$, 
$ \Theta_{y_n}(x_n)=\max_{M} \Theta_{y_n} $. Now, let $ (y_{n_k}) $ be a subsequence of $ (y_n) $ that converges to some $ y\in N $. Such subsequence exists by compactness of $N$.  On the one hand, by the continuity of the map $z\mapsto \max_{M} \Theta_z $, we obtain that $ \Theta_{y_{n_k}}(x_{n_k})=\max_{M} \theta_{y_n} \to 
 \max_{M} \Theta_y $. On the other hand, we have  $ \Theta_{y_{n_k}}(x_{n_k}) \to  \Theta_y(x)$. Therefore $ \Theta_y(x)=\max_{M} \Theta_y $ and thus 
  $ x\in \Gamma$. This proves that $\Gamma$ is closed. \\
  
  \noindent Les us come to the proof of the second part of the statement. Let $x \in \Gamma$. By definition of $\Gamma$, there exists $y_1 \in N$ such that $x$ is a maximum for $\Theta_{y_1}$.    
Assume that $\# \Omega_x \geq 2$ and let also $y_2, y_3 \in  \Om_x $, $y_2 \not= y_3$. Then, $x \in \Om_{y_i}$ for $i=2,3$ and then  $x$ is a minimum of  $\Theta_{y_i}$ $i=2,3$.  By Equation (\ref{critical}),
$y_1,y_2, y_3 \in D_x$. Since $N$ is the boundary of a strictly convex domain, $D_x$ intersects $N$ at at most two points. Since $y_2 \not= y_3$, we must have $y_1=y_2$ or $y_1=y_3$. Let us assume for instance that $y_1 = y_2$. This means that $x$ is a minimum as well as a maximum of $\Theta_{y_1}$ which forces $\Theta_{y_1}$ to be constant on $M$.  By Lemma  \ref{nonconstant}, this cannot occur. \\

\noindent Let us prove now the third part of the statement. For all $x \in M$, it holds that $\varphi^{**}(x)= \psi^*(x) = \varphi(x)$. By Proposition \ref{varphibox}, $\varphi^{\Box \Box} \leq \varphi$. 
 It thus remains to prove that if $x \in \Gamma$ then $\varphi^{\Box \Box}(x) \geq \varphi(x)$. For such $x$, there exists $ y \in N$ such that 
\begin{eqnarray*}
\frac{|x-y|^2}{2} - \varphi(x) 
& = & \varphi^\Box (y) =\varphi^{\Box\Box\Box}(y)\\
& = &  \sup_{z\in M} \Bigl( \frac{|z-y|^2}{2} - \varphi^{\Box\Box}(z)\Bigr) \\
& \ge & \frac{|x-y|^2}{2} - \varphi^{\Box\Box}(x) \\
& \ge & \frac{|x-y|^2}{2} - \varphi(x) 
\end{eqnarray*}
Here, we used the fact that $\varphi^{\Box \Box \Box}= \varphi^{\Box}$, which is proven in Proposition \ref{varphibox}.  We then must have equality in all the inequalities above which implies $\varphi^{\Box \Box}(x) \geq \varphi(x)$.
 \end{proof}
 
\noindent We are now in position to define  $$
\begin{array}{rcl}
T \; :\; \Gamma & \to & N \\
 x& \mapsto & T(x)
 \end{array}
 $$
  such that $\Theta_{T(x)} (x) = \sup_M \Theta_{T(x)} $.
 Then, 
 
\begin{proposition} \label{main2}  {\em (Properties of the mapping $T$)}  
$T$ is a well defined continuous map which is surjective. Moreover, for all $x \in \Gamma$, the outer unit normal vector to $M$ at $ x$  and 
 to $ N $ at $T(x)$ satisfy:
  $$n_M(x) \cdot n_N(T(x)) <0.$$
 
  \end{proposition}

\begin{proof} 
To show that $T$ is well defined, we have to show that for all $x \in \Gamma$, there exists one and only one $y \in N$ such that $\Theta_{y} (x) = \sup_M \Theta_{y}$. The existence of such $y$ is ensured by the fact that $x \in \Gamma$. 
Now, assume that $y_1$ and $y_2$ satisfy this relation. Then, $y_1$ and $y_2$ must belong to the right line $D_x$ (see Relation (\ref{critical})). Moreover,  since $ \Omega_x $ is never empty, let $y_3 \in \Omega_x$ then again $y_3 \in D_x$. Notice that $y_3$ is distinct from $y_1$ and $y_2$ otherwise $\Theta_{y_1}$ is constant which is prohibited by Lemma \ref{nonconstant}. Since $D_x$ intersects $N$ at at most two points, $y_2$ and $y_3$ must be equal. This prove that $T$ is well defined. \\

\noindent The fact that $ T $ is surjective is obvious: if $y \in N$, we choose $x \in M$, which is compact, such that $x$ is a maximum of $\Theta_y$. The definition of $T$ implies that $T(x) = y$. Let us show the continuity of $T$. 
Let $(x_n) \subset \Gamma$ such that $ x_n\to x $ in $ \Gamma$.  By construction we have  
$ \theta_{T(x_n)}(x_n)=\max_{M} \theta_{T(x_n)} $. Now, let $ (x_{n_k}) $ be a subsequence of $ (x_n) $ such that $ T(x_{n_k}) $ converges to some $ y\in N $. Obviously, proceeding as in the proof of Proposition \ref{main1}, $ \Theta_y(x)=\max_{M} \Theta_y $ and thus $ y=T(x) $. Therefore $ T(x) $ is the unique adherence  point of the sequence $ (T(x_n) )$ which 
ensures that $ T(x_n) \to T(x) $ and proves the continuity of $ T$. \\

\noindent Let us prove the last part of the statement of Proposition \ref{main2}. Let $x \in \Gamma$ and set $y=T(x)$. Since $x$ a maximum for $\Theta_y$, by (\ref{critical}), $y \in D_x$. Let also $y'$ be a minimum for $\Theta_x$ i.e. $y' \in \Om_x$. By (\ref{equivom}), $x$ is also a minimum for $\Theta_{y'}$ which implies, by (\ref{critical}), that $y' \in D_x$. Moreover, $y \not= y'$ since otherwise this  would imply that $\Theta_y$ is constant  which would contradict Lemma \ref{nonconstant}. 
Since $N$ bounds a strictly convex domain, $D_x$ intersects $N$ at at most two points which forces to have
$$D_x \cap N = \{ y, y' \}.$$
Note that $n_M(x) \cdot n_N(y) \not= 0$ since otherwise this would imply that $D_x$ is tangent to $N$ and thus intersects $N$ at only one point.  Note also that among $n_M(x) \cdot n_N(y)$, $n_M(x) \cdot n_N(y')$, one is positive and the other one is negative. It then suffices to prove that $n_M(x) \cdot n_N(y')>0$. This fact is proven in \cite{GMC}: with their notations,  $t_+(x)=y'$  satisfies the desired relation.  Since  the proof is easy  we repeat it here for sake of completeness. 
Let $x' \in \Om_y$. Then, $x' \not= x$ otherwise we would get that $\Theta_x$ is constant. 
The {\em monotonicity property} asserts that 
\begin{eqnarray} \label{monoto}
 (x-x') \cdot (y-y') \leq 0. 
\end{eqnarray}
Indeed, since $x \in \Om_{y'}$ and $x' \in \Om_y$, we have 
$$\varphi(x) + \psi(y') = c_2(x,y') \; \hbox{ and } \varphi(x') + \psi(y)= c_2(x',y).$$
Moreover, by definition of $(\varphi, \psi)$, 
$$\varphi(x) + \psi(y) \leq c_2(x,y) \; \hbox{ and } \varphi(x') + \psi(y') \leq c_2(x',y').$$
These relations imply that 
$$c_2(x',y) + c_2(x,y')  \leq c_2(x,y) + c_2(x',y').$$
Coming back to the definition of $c_2$, we obtain Relation (\ref{monoto}). 
Note that since $y,y' \in D_x$, the definition of $D_x$ tells us that  $\overrightarrow{yy'} = \lambda n_M(x)$ for some $\lambda\not=0$. If we assume that $n_M(x) \cdot n_N(y')<0$ and $n_M(x) \cdot n_N(y)>0$ then the fact that $ N $ bounds a strictly convex domain forces $\lambda<0$.  Therefore, Relation 
(\ref{monoto}) becomes $(x-x') \cdot n_M(x)   \leq 0$ which is impossible since $x \not= x'$ and since $M$ bounds a strictly convex domain. This ends the proof of Proposition \ref{main2}.
\end{proof}

Note that the preceding proof shows that 
\begin{lemma} \label{lem1}
 For all $x \in \Gamma$, $D_x \cap N$ has exactly two distinct points $y,y'$ such that 
 $x$ is a maximum for $\Theta_y$ and a minimum for $\Theta_{y'}$. Moreover 
 $$n_M(x) \cdot n_N(y) <0 \; \hbox{ and } \; n_M(x) \cdot n_N(y') >0.$$
\end{lemma}

\begin{proposition} \label{main3} 
 The  following assertions are equivalent: 
 \begin{enumerate} 
  \item $\varphi^{\Box}$ is $C^1$;
\item $T$ is injective; 
\item for all $y \in M$, $\Theta_y$ has exactly two critical points. 
 \end{enumerate}
If one of these assertions is true, then $\Gamma = M$. 
 \end{proposition}

\begin{proof} 
Let us show that 1) implies 2). Assume $\varphi^{\Box}$ is $C^1$. For all $y \in N$, define the right line 
$$ D^\Box_y=y-\nabla \varphi^\Box(y)+\span(n_N(y)) \; .$$ 
An straightforward computation shows that for $x \in M$, 
\begin{eqnarray} \label{critical2}
x \in D^{\Box}_y \Longleftrightarrow y \hbox{ is a critical point of } \Theta^\Box(\cdot):= c_2(x, \cdot) - \varphi^\Box(\cdot).   
\end{eqnarray}
Let $x  \in  \Gamma$. Then $x$ is a  maximum of the function $\Theta_y$ 
i.e. $\varphi^\Box(y) = \Theta_y(x)$ for some $y \in N$. Now, using the fact that  $\varphi^{\Box \Box \Box} = \varphi^{\Box}$ (see Proposition \ref{varphibox}), one  also has
\begin{eqnarray} \label{a1}
 \sup_{z \in M} (c_2(z,y) - \varphi^{\Box \Box} (z)) = \varphi^{\Box \Box \Box}(y) = \varphi^\Box(y) = \Theta_y(x).
\end{eqnarray}
On the other hand, by Proposition \ref{main1}, $\varphi(x) = \varphi^{\Box \Box}(x)$ and hence
$$\Theta_y(x) = c_2(x,y) - \varphi^{\Box \Box}(x).$$
Together with (\ref{a1}), we get that $x$ is a maximum for $z \to c_2(z,y) - \varphi^{\Box \Box} (z)$. Obviously, mimicking what was done to get (\ref{equivom}), we also have that $y$ is a maximum for the function of N
$z \to c_2(x,z) - \varphi^{\Box}(z)$. Relation (\ref{critical2}) then leads to $x \in D^\Box_y$. 
Assume now that $T(x) = T(x')$. We then obtain that $x,x' \in D^\Box_{T(x)}$. Moreover, Lemma \ref{lem1} also establishes that $n_M(x) \cdot n_N(T(x))<0$ and $n_M(x') \cdot n_N(T(x)) <0$ which forces $x$ and $x'$ to be equal since $M$ bounds a strictly convex domain. This proves that $T$ is injective. \\

\noindent Let us prove that 2) implies 3). At first, we show that under assumption 2), $\Gamma = M$.  From Propositions \ref{main1} and \ref{main2}, $T: \Gamma \to M$ is now bijective, continuous. Since $\Gamma$ is compact, it sends closed sets on closed sets and thus $T$ is actually a homeomorphism. This ensures that $\Gamma=M$. Indeed, $M$ and $N$ bound some convex domains in $\R^n$ and are then diffeomorphic to $S^n$. We just proved that $\Gamma$ is a closed set of $M$ homeomorphic to $N$ and hence to $S^n$. To prove that $\Gamma=M$, it suffices to notice that it is open in $M$ and to conclude by the fact that $M$ is connected.   
This follows from  the Jordan-Brouwer separation theorem (see for instance
  \cite{GH}, Corollay (18.9) Page 110). \\

 \noindent We are now in position to prove 3). A consequence of Lemma \ref{lem1} and the fact that $\Gamma = M$ is that for all $(x,y) \in M \times N$ such that $x \in \Om_y$ then 
 \begin{eqnarray} \label{a2}
  n_M(x) \cdot n_N(y) <0.
 \end{eqnarray}
We already noticed that each $\Om_x$ ($x \in M$) is reduced to a point (this comes from Proposition \ref{main1} and the fact that $\Gamma=M$) but this is also true for $\Om_y$ for any $y \in N$. Indeed, if $x,z$ are some minima for $\Theta_y$, they must belong to the right line $D_y$ and they must satisfy (\ref{a2}) which is only possible if $x=z$.  So, let $y \in N$ and  let $x$ be a minimum of $\Theta_y$ and $x'$ be a maximum of $\Theta_y$. Assume that $\Theta_y$ has some other critical point $x''$. Then, $y \in D_{x''}$. By Lemma \ref{lem1}, $x''$ must be a maximum or a minimum of $\Theta_y$. The argument above tells us that $x''$ cannot be a minimum. But $x''$ cannot be either a maximum: it would imply $T(x) = T(x'')$ which is impossible since we assumed $T$ to be injective.  This proves that the only critical points of $\Theta_y$ are $x,x'$.  \\

\noindent Finally, we prove that 3) implies 1). Assume that $\Theta_y$ has only two critical points for any $y \in N$. Then, for all $y$, $\Om^\Box_y$ is reduced to one point (otherwise, $\Theta_y$ has at least two maxima and one minimum). From Point 5) of Proposition \ref{varphibox}, we obtain that $\varphi^\Box$ is differentiable on $N$. It remains to prove that its differential is continuous. Let $(y_k)$ be a sequence of points in $N$ tending to some $y$. Let $x_k \in \Om^\Box_{y_k}$ and $x \in \Om^\Box_y$. Let $x'$ be an adherence point of $(x_k)$.  Since $x_k$ is a maximum of $z \to c_2(z,y_k) - \varphi(z)$,   passing to the limit, $x'$ is a maximum of  $z \to c_2(z,y) - \varphi(z)$. This implies that $x' \in \Om^\Box_y$, which is reduced to one point. Hence $x'=x$ which shows that $x_k$ tends to $x$.  In particular, the sequence of right lines $(D^\Box_{y_k})$  (which are orthogonal to the tangent spaces $T_{y_k} N$ and which are such that $x_k \in  D^\Box_{y_k}$)  has a limit position which is $D^\Box_y$. The definition of these right lines gives  the continuity at $y$ of the differential of $\varphi^\Box$. This ends the proof of Proposition \ref{main3}. 
\end{proof}

\noindent Theorem \ref{main} is a direct consequence of Propositions \ref{main1}, \ref{main2} and \ref{main3}.

\end{document}